\documentclass[leqno,11pt]{amsart}
\usepackage{amssymb,amsmath,amsfonts,amsthm,mathtools,latexsym,ifthen,bezier}
\usepackage{graphicx}
\usepackage{amsxtra}
\usepackage{xspace}

\usepackage{enumitem}
\usepackage[margin=1.3in]{geometry}
\usepackage{setspace}
\usepackage[english]{babel}
\usepackage{bussproofs}
\usepackage{csquotes}
\usepackage{mathtools}
\usepackage{mathrsfs}
\usepackage{latexsym}
\usepackage{enumitem}
\usepackage{amsthm}
\usepackage{amssymb}
\DeclareMathAlphabet{\mathpzc}{OT1}{pzc}{m}{it}
\usepackage[latin1]{inputenc}
\usepackage{afterpage}

\usepackage{bbm}
\usepackage{tikz} 

\newtheorem{theorem}{Theorem}[section]

\newtheorem{lemma}[theorem]{Lemma}
\newtheorem{proposition}[theorem]{Proposition}
\newtheorem{corollary}[theorem]{Corollary}

\newtheorem{fact}[theorem]{Fact}
\newtheorem{claim}[theorem]{Claim}
\newtheorem{maintheorem}[theorem]{Main Theorem}
\theoremstyle{definition}
\newtheorem{definition}[theorem]{Definition}

\theoremstyle{remark}
\newtheorem{remark}{Remark}
\newtheorem{notation}{Notation}
\newtheorem{question}{Question}

\def\hook{\upharpoonright}
\def\forces{\Vdash}

\newfont{\ssi}{cmssi12 at 12pt}

\newenvironment{ea*}{\begin{eqnarray*}}{\end{eqnarray*}}

\setbox0=\hbox{$\longrightarrow$}

\newcommand{\calA}{\mathcal{A}}
\newcommand{\calB}{\mathcal{B}}

\renewcommand{\phi}{\varphi}

\newcommand{\ZFC}{\ensuremath{\mathsf{ZFC}}\xspace}
\newcommand{\ZF}{\ensuremath{\mathsf{ZF}}\xspace}

\newcommand{\proves}{\vdash}

\def\<#1>{\langle#1\rangle}

\renewcommand{\P}{{\mathord{\mathbb P}}}
\newcommand{\Q}{{\mathord{\mathbb Q}}}

\newcommand{\MP}{\ensuremath{\mathsf{MP}}}

\newcommand{\ColNothing}{\mathrm{Col}}
\newcommand{\Col}[1]{\ColNothing(#1)}

\newcommand{\MPColNothing}[1]{\MP_{\Col{\dot{\kappa}}}}

\newcommand{\CH}{\ensuremath{\mathsf{CH}}\xspace}

\def\hook{\upharpoonright}
\def\forces{\Vdash}

\def\Me{\mathcal M}
\def\Null{\mathcal N}
\def\ZFC{\mathsf{ZFC}}
\def\ZF{\mathsf{ZF}}

\def\baire{\omega^\omega}

\def\mfc{\mathfrak{c}}

\def \mfd{\mathfrak{d}}
\def\mfa{\mathfrak{a}}
\def\mfu{\mathfrak{u}}
\def\mfi{\mathfrak{i}}

\def\CH{\mathsf{CH}}

\def\calA{\mathcal A}
\def\scrU{\mathscr U}
\def\proves{\vdash}
 
\begin{document}

\title{Projective Well Orders and Coanalytic Witnesses}

\author[Bergfalk]{Jeffrey Bergfalk}
\address[J. ~Bergfalk]{Institut f\"{u}r Mathematik, Kurt G\"odel Research Center, Universit\"{a}t Wien, Kolingasse 14-16, 1090 Wien, AUSTRIA}
\email{jeffrey.bergfalk@univie.ac.at}

\author[Fischer]{Vera Fischer}
\address[V. ~Fischer]{Institut f\"{u}r Mathematik, Kurt G\"odel Research Center, Universit\"{a}t Wien, Kolingasse 14-16, 1090 Wien, AUSTRIA}
\email{vera.fischer@univie.ac.at}

\author[Switzer]{Corey Bacal Switzer}
\address[C.~B.~Switzer]{Institut f\"{u}r Mathematik, Kurt G\"odel Research Center, Universit\"{a}t Wien, Kolingasse 14-16, 1090 Wien, AUSTRIA}
\email{corey.bacal.switzer@univie.ac.at}

\thanks{\emph{Acknowledgements:} The authors would like to thank the
Austrian Science Fund (FWF) for the generous support through grant number Y1012-N35.}

\maketitle

\begin{abstract}
We further develop a forcing notion known as \emph{Coding with Perfect Trees} and show that this poset preserves, in a strong sense, definable $P$-points, definable tight MAD families and definable selective independent families. As a result, we obtain a model in which $\mathfrak{a}=\mathfrak{u}=\mathfrak{i}=\aleph_1<2^{\aleph_0}=\aleph_2$,  each of $\mathfrak{a}$, $\mathfrak{u}$, $\mathfrak{i}$ has a $\Pi^1_1$ witness and there is a $\Delta^1_3$ well-order of the reals. Note that both the complexity of the witnesses of the above combinatorial cardinal characteristics, as well as the complexity of the well-order are optimal. In addition, we show that the existence of a $\Delta^1_3$ well-order of the reals is consistent with $\mathfrak{c}=\aleph_2$ and each of the following: $\mathfrak{a}=\mathfrak{u}<\mathfrak{i}$, $\mathfrak{a}=\mathfrak{i}<\mathfrak{u}$, $\mathfrak{a}<\mathfrak{u}=\mathfrak{i}$, where the smaller cardinal characteristics have co-analytic witnesses. 

Our methods allow the preservation of only sufficiently definable witnesses, which significantly differs from other preservation results of this type.
\end{abstract}

\section{Introduction}
In recent literature an important line of research on the border between combinatorial and descriptive set theory has developed by looking at various combinatorial sets of reals of optimal projective complexity. Beginning with Miller's seminal \cite{Millerpi11}, this includes the study of witnesses to various cardinal characteristics defined as the maximal set of reals with such and such property; see also \cite{DefMIF}, \cite{FFST20}. Another important example is the study of the possible complexity of well-orderings of reals; see \cite{FF10} and the references therein.

In the present article we contribute to this project by exhibiting a model where a number of different combinatorial sets of reals all have optimal projective complexity simultaneously. Specifically our main theorem is as follows.

\begin{maintheorem}
It is consistent that $\mfa= \mfu = \mfi = \aleph_1 < 2^{\aleph_0} = \aleph_2$, $\mfu$, $\mfa$, $\mfi$ have $\Pi^1_1$ (and hence optimal) witnesses of cardinality $\aleph_1$ and there is a $\Delta^1_3$ well-order of the reals (also optimal).
\label{mainthm1}
\end{maintheorem}

The main techniques involved in proving Main Theorem \ref{mainthm1} are the forcing from \cite{FF10} for forcing $2^{\aleph_0} = \aleph_2$ alongside a $\Delta^1_3$ well-order and the study of stronger versions of maximality for almost disjoint families, filter bases and independent families and proving preservation theorems relating to these which can be used in our context. An interesting feature of our analysis is that, while preservation theorems for these classes already existed, we often had to modify them, and restrict to ``definable" witnesses in order to preserve any maximal set along our intended iterations. A by-product of our analysis is that we can also arrange a variety of cardinal characteristic constellations to hold alongside optimal witnesses.


\begin{maintheorem}
The following cardinal characteristic constellations are consistent with a $\Delta^1_3$ well-order of the reals, $2^{\aleph_0} = \aleph_2$ and all cardinal characteristics of size $\aleph_1$ below having $\Pi^1_1$ witnesses:
\begin{enumerate}
\item
$\mfa = \mfu < \mfi=\mathfrak{c}$,
\item
$\mfa = \mfi < \mfu=\mathfrak{c}$,
\item
$\mfa < \mfu = \mfi=\mathfrak{c}$.
\end{enumerate}
\end{maintheorem}

The study of the three cardinals $\mfi$, $\mfu$ and $\mfa$ involves studying various witnesses wherein the maximality condition is strengthened in such a way that it can be preserved by countable support iterations of $S$-proper posets for a suitably chosen stationary $S\subseteq \omega_1$. It seems that our methods only preserve witnesses of this type that are sufficiently definable, a circumstance which is, to our knowledge, different from that of other preservation theorems of this sort. 

The rest of this paper is organized as follows. In the next section we review the main forcing from \cite{FF10} and modify it slightly in a manner important for our purposes. Sections 3-5 contain our preservation theorems leading into Section 6 where applications, including the proof of Main Theorem \ref{mainthm1}, are given. Section 7 concludes the paper with some remarks, open questions and lines for future research.

\section{The Fischer-Friedman Forcing to Add a $\Delta^1_3$-Well Order}
In this section we record a slight variation on the main forcing from \cite{FF10}. This forcing, which is defined over $L$, adds $\Delta^1_3$-well order of the reals, forces $2^{\aleph_0} = \aleph_2$ and has the Sacks property. Moreover, it is flexible in the sense that other proper forcing notions can be ``woven in" along the iteration resulting in various constellations of cardinal characteristics. See \cite[Section 6]{FF10} for several theorems of this form. Here, for the most part, we go over the main attributes of this forcing for the convenience of the reader. Most proofs are left out and we refer the reader to \cite{FF10} for a more detailed discussion of this poset. However, for our preservation theorems in Sections 3-5 we will need to slightly augment the definition of the Sacks coding iterand (Subsection 2.3 below) and this augmentation will be discussed in more detail.

There are three ingredients that go into the Fischer-Friedman forcing and we take them one at a time. Throughout the rest of the paper, following \cite{FF10} we say that a transitive model $\Me$ is {\em suitable} if $\omega_2^\Me$ exists and $\omega_2^\Me = \omega_2^{L^\Me}$. The forcing notion we define will be a countable support iteration $\P_{\omega_2}$ of forcing notions and each iterand will be defined in some $L[G^*]$ where $G^*$ is generic for some forcing notion $\P^* \in L$ in which cofinalities (and hence cardinals) have not been changed. Consequently we assume when making definitions in this section that $V= L[G^*]$ and the cofinalities of $V$ are the same as those of $L$.

\subsection{Localization}
Let $X \subseteq \omega_1$ and $\varphi(\omega_1, X)$ be a $\Sigma_1$ sentence with parameters $\omega_1$ and $X$ which is true in all suitable models which contain $\omega_1$ and $X$ as elements. The first forcing notion we will need to define $\P_{\omega_2}$ is the following.

\begin{definition}[Localization Forcing]
Let $\mathcal L(\varphi)$ be the partial order consisting of all functions $r: |r| \to 2$ where $|r|$ is a countable limit ordinal and the following hold.
\begin{enumerate}
\item
If $\gamma < | r|$ then $\gamma \in X$ if and only if $r(2\gamma) = 1$.
\item
If $\gamma \leq |r|$ and $\Me$ is a countable, suitable model containing $r \hook \gamma$ as an element and $\gamma = \omega_1^\Me$ then $\Me \models \varphi(\gamma, X \cap \gamma)$. 
\end{enumerate}
The extension relation is end extension.
\end{definition}

Observe that if $r \in \mathcal L(\varphi)$ then the even part of $r$ codes $X \cap |r|$. The following facts from \cite{FF10} are relevant for our discusssion.

\begin{fact}
The following hold.
\begin{enumerate}
\item (\cite[Lemma 1]{FF10})
For all $\gamma < \omega_1$ the set of conditions $r \in \mathcal L(\varphi)$ so that $\gamma \leq |r|$ is dense.
\item(\cite[Lemma 2]{FF10})
Let $G \subseteq \mathcal L(\varphi)$ be generic over $V$ and let $Y = \bigcup G$. If $\Me$ is a suitable model containing $Y \hook \omega_1^\Me$ as an element then $\Me \models \varphi(\omega_1^\Me, X \cap \omega_1^\Me)$.
\item (\cite[Lemmas 3 and 4]{FF10})
$\mathcal L(\varphi)$ has a $\sigma$-closed dense subset and therefore is proper and adds no reals.
\end{enumerate}
\end{fact}
\subsection{Club Shooting}
Given a stationary, co-stationary set $S \subseteq \omega_1$ recall the forcing $Q(S)$ for shooting a club through the complement of $S$. Namely $Q(S)$ consists of countable closed subsets of $\omega_1 \setminus S$ ordered by end-extension. The following is well-known, see \cite{JechST}.

\begin{fact}
For any stationary, co-stationary $S \subseteq \omega_1$ the forcing $Q(S)$ is $\omega$-distributive and $\omega_1 \setminus S$-proper.
\end{fact}

Here, a forcing notion $\Q$ is $T$-{\em proper} for some stationary $T \subseteq\omega_1$ if for any sufficiently large $\theta$ and countable elementary submodel $\Me \prec H_\theta$  such that $\Me \cap \omega_1 \in T$, every $p \in \Q$ extends to a $q \leq p$ which is $(\Me, \Q)$-generic. It is well known that most iteration and preservation theorems for countable support iterations of proper notions apply \emph{mutatis mutandis} to countable support iterations of $T$-proper forcing notions.

We will need a special sequence of stationary sets $S$ for which the forcings of the form $Q(S)$ will appears as iterands in $\P_{\omega_2}$. We describe these now.

\begin{lemma}[See Lemma 14 of \cite{FF10}]
Assume $V= L$. There is a function $F:\omega_2 \to L_{\omega_2}$ which is $\Sigma_1$ definable over $L_{\omega_2}$ and a sequence $\vec{S} = \langle S_\beta \; | \; \beta < \omega_2\rangle$ of almost disjoint stationary subsets of $\omega_1$ which is $\Sigma_1$-definable over $L_{\omega_2}$ with parameter $\omega_1$ so that $F^{-1}(a)$ is unbounded in $\omega_2$ for every $a \in L_{\omega_2}$ and whenever $\Me$ and $\Null$ are suitable models so that $\omega_1^\Me = \omega_1^\Null$ then $F^\Me$ and $\vec{S}^\Me$ agree with $F^\Null$ and $\vec{S}^\Null$ on $\omega_2^\Null \cap \omega_2^\Me$. In addition, if $\Me$ is suitable and $\omega_1^\Me = \omega_1$ then $F^\Me$ and $\vec{S}^\Me$ equal the restrictions of $F$ and $\vec{S}$ to the $\omega_2$ of $\Me$.
\label{FandS}
\end{lemma}

From now on, fix an $F$ and $\vec{S}$ as described in Lemma \ref{FandS}.

\subsection{Coding with Perfect Trees}
The third component of  $\P_{\omega_2}$ is the only one which adds reals. It is a variation of Sacks forcing called the {\em coding with perfect trees} forcing, denoted by $C(Y)$. To define it, we assume that $V = L[Y]$ where $Y \subseteq \omega_1$ is generic for some forcing in $L$ which does not change cofinalities. Fix a sequence $\vec{\mu}=\langle\mu_i \; | \; i < \omega_1\rangle$ defined inductively so that $\mu_i > \bigcup_{j < i} \mu_j$ is for each $i < \omega_1$ a countable ordinal which is least with the property that $L_{\mu_i}[Y\cap i]$ is $\Sigma^1_5$ elementary in $L_{\omega_1}[Y \cap i]$ and is a model of $\ZF^- + $``$\omega$ is the largest cardinal".  Note that this is slightly different than the definition of $\vec{\mu}$ given in \cite{FF10}, which requires no elementarity whatsoever. However, requiring a degree of elementarity at any fixed level of the projective hierarchy does not change the fact that the sequence $\vec{\mu}$ is definable in $L_{\omega_1}[Y]$, and this is all that matters for their arguments.

\begin{remark}
The use of $\Sigma^1_5$ as opposed to say, $\Sigma^1_6$, $\Sigma^1_{25}$ or $\Sigma^1_{35469}$ is inconsequential. All that matters is that we fix a finite level of the projective hierarchy ahead of time and restrict our arguments to statements at that level or below. The witnesses we will be considering in this paper are mostly $\Pi^1_1$ (and occasionally $\Sigma^1_2$) and it takes two to three more quantifiers to write down the relevant properties of them, hence the choice of 5. However, if we were interested in preserving, e.g., $\Pi^1_5$ witnesses then $\Sigma^1_9$ ($9 = 5 + 4$) would suffice. Since the exact level does not matter, in the interests of this paper's readability we won't count quantifiers too closely, regarding this sort of analysis as essentially routine. Rather we will refer, somewhat excessively, to ``elementarity", simply, and leave the details to the particularly exacting reader.
\end{remark} 

The conditions of $C(Y)$ will be perfect trees $T \subseteq 2^{<\omega}$. For $T$ a perfect tree let $|T|$ be the least $i$ with $T \in L_{\mu_i}[Y \cap i]$. Let us also denote $\mathscr A_i := L_{\mu_i}[Y \cap i]$ for each $i$. A real $r$ {\em codes $Y$ below $i$} if and only if $j \in Y$ if and only if $\mathscr A_j[r] \models \ZF^-$, for all $j < i$. 

\begin{definition}
The forcing notion $C(Y)$ is the set of all perfect trees $T$ so that each branch of $T$ codes $Y$ below $|T|$. The order is inclusion.
\end{definition}

We need some standard notation concerning trees. Recall that if $T \subseteq 2^{<\omega}$ (or $\omega^{<\omega}$) is a tree then a  node $t \in T$ is called {\em splitting} if it has more than one immediate sucessor. Denote by ${\rm Split}(T)$ the set of splitting nodes in $T$. For $n < \omega$ we say that $t \in T$ is $n$-{\em splitting} if it is splitting and has exactly $n-1$ predecessors which are splitting. We denote by ${\rm Split}_n(T)$ the set of $n$-splitting nodes. Note ${\rm Split}(T) = \bigcup_{n < \omega} {\rm Split}_n(T)$. Given conditions $p, q \in C(Y)$ we say that $p \leq_n q$ if $p \leq q$ and ${\rm Split}_n(p) = {\rm Split}_n(q)$. 

\begin{fact}
The following hold.
\begin{enumerate}
\item
(\cite[Lemma 5]{FF10}) If $T \in C(Y)$ and $|T| \leq i < \omega_1$ then there is a $T' \leq T$ so that $|T'| = i$.
\item (\cite[Lemma 7]{FF10}) $C(Y)$ is proper. 
\item (\cite[Lemma 6]{FF10}) If $G \subseteq C(Y)$ is generic over $V$ and $R = \bigcap G$ is a real and $R$ codes $Y$. In other words for all $j < \omega_1$ we have $j \in Y$ if and only if $L_{\mu_j}[Y \cap j, R] \models \ZF^-$.
\end{enumerate}
\end{fact}

Moving forward the following strategy will be used frequently to find new conditions. We will begin with some condition $q$ so that $|q| = \delta$ for some $\delta < \omega_1$. Note that $q \in \calA_\delta$ by definition. By the definition of the forcing $C(Y)$, every branch through $q$ will code $Y$ up to $\delta$. If $q ' \subseteq q$ is a perfect subtree of $2^{<\omega}$ which is in $\calA_\delta$ then $q' \in C(Y)$. This is because, since $q' \in \calA_\delta$, $|q'| \leq \delta$ and, since every branch of $q$ and hence of $q'$ code $Y$ up to $\delta$ we must have that $q'$ is a condition. This type of argument is essential in many of our results.

In order to facilitate our discussion later we prove a fact about $C(Y)$ now that will be useful in Sections 3-5. If $\dot{X}$ is a $C(Y)$-name, $p \in C(Y)$ and $p \forces \dot{X} \subseteq \omega$ then the {\em outer hull} of $\dot{X}$ with respect to $p$, denoted $X_p$, is the set $\{m \; | \; p \nVdash \check{m} \notin \dot{X}\}$. If $t \in {\rm Split}(p)$ then we let $X^p_t$ denote the outer hull of $\dot{X}$ with respect to $p_t$. A condition $p$ is {\em preprocessed for} $\dot{X}$ if for each $n < \omega$ and each $n$ split node $t$ we have that $p_t$ decides $\dot{X} \cap \check{n}$.

\begin{lemma}
For every $\dot{X}$ the set of conditions preprocessed for $\dot{X}$ is dense in $C(Y)$. In fact, for every $k < \omega_1$ and every $p \in C(Y)$ there is a $j > k$ and a preprocessed $q \leq p$ so that $q \in \mathscr A_j$ and for all $x \in [q]$ $x$ codes $Y$ below $j$. Moreover $\mathscr A_j$ contains the function $i:{\rm Split}(q) \to [\omega]^{<\omega}$ defined by letting $i(t)$ be, for each $n < \omega$ and $t \in {\rm Split}_n(q)$, the finite set $a_t$ such that $q_t$ forces $\dot{X} \cap \check{n} = \check{a}_t$.
\label{preprocessed}
\end{lemma}

\begin{proof}
Fix $p$ and $\dot{X}$. Let $M \prec L_{\omega_2}[Y]$ be countable containing $p, \dot{X}$ and let $\bar{M}$ be its transitive collapse. As explained in \cite{FF10} if $\delta = (\omega_1)^{\bar{M}}$ then $\bar{M} \in \mathscr A_\delta$ (and hence $p, \dot{X} \in \calA_\delta$). Since $\calA_\delta$ thinks that $\delta$ is countable, but $\bar{M}$ does not, there is a countable, strictly increasing, cofinal sequence of ordinals $\vec{\delta} := \langle \delta_n \; | \; n < \omega\rangle \in \calA_\delta$ whose supremum is $\delta$ and $\vec{\delta} \subseteq\bar{M}$. Now in $\bar{M}$ it's clear that we can find, for each $n < \omega$ and each $r \in C(Y)$ a condition $r' \leq_n r$ so that $|r'| \geq \delta_n$ and for each $n$ split node $t$ of $r'$ we have that $(r')_t$ decides $\dot{X} \cap \check{n}$. Applying this observation iteratively from the point of view of $\calA_\delta$ yields a a fusion sequence $p = p_0 \geq _1 p_1 \geq_2 p_2 \geq_3 p_3 ...$ so that, letting $p_\omega = \bigcap_{n < \omega} p_n \in C(Y) \cap \calA_\delta$, we have $|p_\omega| \leq \delta$, each branch of $p_\omega$ codes $Y$ below $\delta$ and for each $n<\omega$ and each $n$-splitting node $t$ of $p_\omega$, $(p_\omega)_t$ decides $\dot{X} \cap \check{n}$. Moreover, since $\bar{M}$ knows which finite set each $(p_\omega)_t$ decides for $\dot{X} \cap \check{n}$,  $\mathscr A_j$ can record this data during the construction of the fusion sequence and hence define $i$. 
\end{proof}

\subsection{Putting it All Together}
We're now ready to define $\P_{\omega_2}$. Recall we fixed $F$ and $\vec{S}$ as in Lemma \ref{FandS}. Assume $V= L$ and fix an $S \subseteq \omega_1$ which is almost disjoint from every $T \in \vec{S}$. Our countable support iteration $\langle \P_\alpha, \dot{\Q}_\alpha \; | \; \alpha < \omega_2\rangle$ is defined making use of the following vocabulary and notation. We can assume that all names for reals are {\em nice} in the sense that if $\dot{f}$ is an $\mathbb H$-name for a real for some $\mathbb H$ then $\dot{f}$ has the form $\bigcup_{i \in \omega}\{\langle \langle i, j_i\rangle \check{}, p \rangle \; | \; p \in A_i(\dot{f})\}$ where $A_i(\dot{f})$ is a maximal antichain of elements deciding $\dot{f}(\check{i})$. Assume moreover that for all $\alpha < \beta < \omega_2$ all $\P_\alpha$-names for reals appear before all $\P_\beta$-names for reals which are note $\P_\alpha$-names for reals in the canonical well order $<_L$ of $L$. For each $\alpha< \omega_2$ define $<_\alpha$ on the reals of $L[G_\alpha]$ (where $G_\alpha$ is $\P_\alpha$ generic) as follows. For each real $x \in L[G_\alpha]$ let $\gamma_x \leq \alpha$ least so that there is a nice $\P_\gamma$-name $\sigma_x^{\gamma_x}$ for $x$ and let $x <_\alpha y$ if and only if $\gamma_x < \gamma_y$ or $\gamma_x = \gamma_y = \gamma$ and $\sigma_x^\gamma <_L \sigma_y^\gamma$. If $G\subseteq \P_{\omega_2}$ is generic over $L$ then $<_G = \bigcup_{\alpha<\omega_2}  <_\alpha$ will be the desired well order. If $x, y \in L[G_\alpha]$ are reals and $x <_\alpha y$ let $x * y = \{2n \; | \; n \in x\} \cup \{2n + 1 \; | \; n \in y\}$. 

The iterated forcing $\P_{\omega_2}$ is defined recursively as follows. Let $\P_0$ be the trivial poset. Assume $\P_\alpha$ has been defined. Let $\dot{\Q}_\alpha = \dot{\Q}_\alpha^0 * \dot{\Q}^1_\alpha$ be a $\P_\alpha$-name for a poset so that $\dot{\Q}_\alpha^0$ is a proper forcing notion of cardinality at most $\aleph_1$ and $\dot{\Q}^1_\alpha$ is defined as follows. If $F(\alpha)$ is not of the form $\{\sigma^\alpha_x, \sigma^\alpha_y\}$ for some $x <_\alpha y$ in $L[G_\alpha]$, let $\dot{\Q}_\alpha^1$ be a $\P_\alpha * \dot{\Q}^0_\alpha$-name for a trivial poset. Otherwise $F(\alpha) = \{\sigma^\alpha_x, \sigma^\alpha_y\}$ for some reals $x, y \in L[G_\alpha]$. In this case we let $\dot{\Q}^1_\alpha$ be a $\P_\alpha * \dot{\Q}^0_\alpha$-name for a three step forcing notion $\mathbb K^0_\alpha * \dot{\mathbb K}^1_\alpha * \dot{\mathbb K}^2_\alpha$ where
\begin{enumerate}
\item
$\mathbb K^0_\alpha$ is a countable support iteration of forcing notions of the form $Q(S_i)$ where $i = \alpha + 2n$ for $n \in x * y$ and $i = \alpha + 2n + 1$ for $n \notin x * y$.
\item
In $V^{\P_\alpha * \dot{\Q}^0_\alpha * \mathbb K^0_\alpha}$ $\dot{\mathbb K}^1_\alpha$ names a localization poset $\mathcal L(\varphi_\alpha)$ where $\varphi_\alpha (\omega_1, X_\alpha)$ is a $\Sigma_1$ sentence, $X_\alpha \subseteq \omega_1$ codes $\alpha$, the pair $(x, y)$, a level of $L$ of size $\leq \aleph_1$, and the generic for $\P_\alpha * \dot{\Q}^0_\alpha * \mathbb K^0_\alpha$ in $V^{\P_\alpha * \dot{\Q}^0_\alpha * \mathbb K^0_\alpha}$ and $\varphi_\alpha (\omega_1, X_\alpha)$ says that $X_\alpha$ codes an ordinal $\bar{\alpha} < \omega_1$ and a pair $(x, y)$ so that $S_{\bar{\alpha} + 2n}$ is nonstationary for $n \in x * y$ and $S_{\bar{\alpha} + 2n + 1}$ is nonstationary for $n \notin x * y$.
\item
By the properties of the localization forcing, in $V^{\P_\alpha * \dot{\Q}^0_\alpha * \mathbb K^0_\alpha * \dot{\mathbb K}_\alpha^1}$ the generic can be coded by a single subset of $\omega_1$, call it $Y_\alpha$. Thus the ground model is of the form $L[Y_\alpha]$. Let $\dot{\mathbb K}^2_\alpha$ name $C(Y_\alpha)$ in this model.
\end{enumerate}
 With this list we complete the inductive definition of $\P_{\omega_2}$. 

\begin{lemma}[Lemma 15 of \cite{FF10}]
$\P_{\omega_2}$ is $S$-proper, $\omega_2$-c.c. and forces $2^{\aleph_0} = \aleph_2$.
\end{lemma}

\begin{fact}[Lemma 3.7 of \cite{FFK14}]
$C(Y)$ has the Sacks property and therefore if each $\dot{\mathbb Q}_\alpha^0$ has the Sacks property (Laver property, $\baire$-bounding etc) then so does $\P_{\omega_2}$.
\end{fact}

The point is the following, which is the the first main result of \cite{FF10}.

\begin{theorem}[Theorem 1 of \cite{FF10}]
Forcing with $\P_{\omega_2}$ adds a $\Delta^1_3$ well-order of the reals.
\end{theorem}

\begin{proof}[Proof Sketch]
Let $G\subseteq \P_{\omega_2}$ be generic over $L$ and work in $L[G]$. In the proof of \cite[Theorem 1]{FF10} it is shown that the following holds:

\begin{center}
For all reals $x, y$ we have $x <_G y$ if and only if there is a real $R$ such that for every suitable, countable model $\Me$ containing $R$ there is an ordinal $\alpha < \omega_2^\Me$ so that $S^\Me_{\alpha + 2n}$ is nonstationary in $\Me$ for $n \in x * y$ and $S^\Me_{\alpha + 2n + 1}$ is nonstationary in $\Me$ for $n$ not in $x * y$. 
\end{center}

Counting quantifiers one can see that this definition is $\Sigma^1_3$ however, since $<_G$ is a total order on the reals, we have that $\nleq_G$ is also $\Sigma^1_3$.
\end{proof}

Let us note that the use of $\dot{\Q}_\alpha^0$ is unnecessary for forcing the $\Delta^1_3$ well order. However, this iterand represents additional forcing notions ``woven into" the iteration, allowing more flexibility in the final model, for example forcing various cardinal characteristic inequalities.

Before moving on to the next section let us make some brief remarks about the forcing notion $\P_{\omega_2}$. Observe that the only iterands which add reals are those of the form $C(Y)$ and, potentially, $\dot{\Q}_\alpha^0$. In what follows, we will be interested in proving preservation properties for $\P_{\omega_2}$. Many times this will take the form of showing that some property holds of forcing notions of the form $C(Y)$ from which, combined with the fact that the other iterands do not add reals (or, in the case of the $\dot{\Q}_\alpha^0$'s are assumed to have some property), we will be able to conclude the property holds of $\P_{\omega_2}$. The first such property we will consider is that of preserving P-points.

\section{Preservation of Definable P-Points}

Recall that a P-point is an ultrafilter $\mathscr U$ with the property that given any countable set $\mathcal A \subseteq \mathscr U$ there is an $X \in \scrU$ so that $X \subseteq^* A$ for all $A \in \mathcal A$. It's known that if $V = L$ there is a P-point $\scrU$ which has a $\Pi^1_1$ base $\hat{\scrU}$ so that $\ZFC \proves$``$\hat{\scrU} \subseteq L$" see \cite{Schilhanultrafilter}. Fix such a $\scrU$. The main result of this section is the following.

\begin{lemma}
If $L[Y] \models$``$\scrU$ is a P-point" then the poset $C(Y)$ forces that $\scrU$ is a P-point.
\label{mainlemmappoint}
\end{lemma}

Given this lemma, alongside the preservation of P-points along countable support iterations of ($S$)-proper posets, \cite[Theorem 21.11]{Hal17} we get the following.

\begin{theorem}
If $G\subseteq\P_{\omega_2}$ is generic over $L$ then in $V[G]$ there is a $\Pi^1_1$ ultrafilter base for a P-point of size $\aleph_1$. In particular, it's consistent that there is a $\Pi^1_1$-ultrafilter base for a P-point, a $\Delta^1_3$ well order of the reals and $\mathfrak{u} = \aleph_1 < \mfc = \aleph_2$.
\label{ppointthm}
\end{theorem}

Towards proving Lemma \ref{mainlemmappoint} work in $L[Y]$ where $Y \subseteq \omega_1$ was added by some forcing in $L$. Note that if $L_\alpha[Y] \prec_{\Sigma^1_5} L_{\omega_1}[Y]$ then the fact that $\hat{\scrU}$ is an ultrafilter base for a P-point is expressible in $L_{\alpha}[Y]$ and will be true by $\Sigma^1_5$ (in fact $\Pi^1_3$) elementarity. This is because $\hat{\scrU}$ is a P-point base if and only if $$\forall \{X_n\; | \; n < \omega\} [ (\exists n< \omega \, \forall Y \in \hat{\scrU} \, Y \nsubseteq^* X_n) \lor \exists Y \, (Y \in \hat{\scrU} \land \forall n < \omega \, Y \subseteq^* X)].$$ 

This is the type of quantifier counting we referred to in Remark 1 above.

We need to first investigate {\em interpretations of} $\dot{X}$, see \cite[p.362]{BarJu95} where a similar idea is introduced for Miller forcing.
\begin{definition}
Given a condition $p \in C(Y)$ a $p$-{\em interpretation of} $\dot{X}$ is a set $I_p(\dot{X}) \in [\omega]^\omega$ so that for each $n < \omega$ there is a $q \leq p$ forcing that $\dot{X} \cap \check{n} = \check{I}_p(\dot{X}) \cap \check{n}$. 
\end{definition}

\begin{proposition}
If $q$ is preprocessed for $\dot{X}$ then there is a $q$-interpretation of $\dot{X}$ which is definable from $q$ in any $\mathscr A_\delta \ni q$ which has the function $i$ described in Lemma \ref{preprocessed}.
\end{proposition}

\begin{proof}
Fix $q$ preprocessed and let $l$ be the leftmost branch of $q$. Now define $I_q(\dot{X}) = \bigcap_{n < \omega} X^q_{l\hook n}$. To see this is an interpretation, fix $n$ and let $s$ be the $n^{\rm th}$ splitting node of $q$ in $l$. Since $q$ is preprocessed, $q_s$ decides $\dot{X} \cap \check{n}$, and $I_q(\dot{X}) \cap n$ agrees with it on this finite set hence $q_s \forces \dot{X} \cap \check{n} = \check{I}_q(\dot{X})$ as needed.

For the definability part, note that if $q \in \mathscr A_\delta$ is preprocessed and constructed as in Lemma \ref{preprocessed} then $I_q(\dot{X})$ is definable in $\mathscr{A}_\delta$ since $k \in I_q(\dot{X})$ if and only if for the $n$-splitting node $t$ of $q$ in $l$ we have that $k \in i(t)$. 
\end{proof} 

We will refer to such a $q$-interpretation as the {\em leftmost interpretation} of $\dot{X}$ and denote it $L_q(\dot{X})$. 

\begin{proposition}
If $q \leq p$ and $I$ is a $q$-interpretation of $\dot{X}$ then $I$ is a $p$-interpretation of $\dot{X}$. Consequently, every condition admits an interpretation of $\dot{X}$.
\end{proposition}

\begin{proof}
The ``consequently part" follows from the previous proposition plus the density of preprocessed conditions. To see the first part, observe that if $I$ is a $q$-interpretation then for each $n$ there is an $r \leq q$ forcing $\dot{X} \cap \check{n} = \check{I} \cap \check{n}$ and hence the same is true of $p$ as needed.
\end{proof}


Before turning to the proof of Lemma \ref{mainlemmappoint} we need to recall one more idea: the P-point game, see \cite[4.4.4, p. 225]{BarJu95}.
\begin{definition}
Let $\mathscr V$ be an ultrafilter on $\omega$. The P-point game for $\mathscr V$, denoted $G(\mathscr V)$ is the following two player game played in $\omega$ many rounds: at stage $n$ player I chooses a set $A_n \in \mathscr V$ and player II responds by playing a finite $a_n \subseteq A_n$. In the end Player I wins if $\bigcup_{n < \omega} a_n \notin \mathscr V$.
\end{definition}

\begin{fact}[Galvin and Shelah, see \cite{BarJu95}, Theorem 4.4.4]
For any ultrafilter $\mathscr V$, player I has a winning strategy for $G(\mathscr V)$ if and only if $\mathscr V$ is not a P-point.
\end{fact}

We now prove Lemma \ref{mainlemmappoint}.

\begin{proof}[Proof of Lemma \ref{mainlemmappoint}]
We work in $L[Y]$, where $Y \subseteq \omega_1$ is generic over $L$ for some forcing notion and assume that $L[Y] \models$``$\mathscr U$ is a P-point". There are a few preliminaries to note. First, as noted at the beginning of this section, since $\hat{\scrU}$ is $\Pi^1_1$ definable in $L$, it is a definable subset of $L_{\omega_1}[Y]$ (but not an element). However $L_{\omega_1} [Y]$ computes it correctly in the sense that $X \in \hat{\scrU}$ if and only if $L_{\omega_1}[Y] \models X \in \hat{\scrU}$. Also, being an ultrafilter base for a P-point is expressible in $L_{\omega_1}[Y]$ and true by assumption. Moreover these properties are preserved downwards in the sense that for each $i < \omega_1$ they all hold in $L[Y \cap i]$ since we require enough projective elementarity. This is a consequence of the facts that $\hat{\mathscr{U}}$ is provably a subset of $[\omega]^\omega \cap L$ and hence unchanged in forcing extensions, the fact that $L[Y \cap i]$ is a forcing extension because of the intermediate model theorem, and the fact that being an ultrafilter base for a P-point is downwards absolute: any counterexample in a smaller model remains a counterexample in a bigger model. The same facts are all also true of $G(\scrU)$ and in particular $G(\scrU)$ is definable in every $\mathscr A_i$ and each $\mathscr A_i$ knows that, since $\scrU$ is a P-point, no strategy for player I is winning. 

Fix $p \in C(Y)$ and $\dot{X}$ a $C(Y)$-name so that $p \forces \dot{X} \in [\omega]^{\omega} \, \check{}$. We need to find a $r\leq p$ so that either $r \forces \dot{X} \cap \check{A} = \emptyset$ for every $A \in \mathscr U$ or else find an $A \in \hat{\scrU}$ so that $r \forces \check{A} \subseteq \dot{X}$. First we apply Lemma \ref{preprocessed} to find a $q \leq p$ which is preprocessed and a $\delta < \omega_1$ so that $q \in \mathscr A_\delta$ and every branch of $q$ codes $Y$ below $\delta$. We will describe a strategy in $G(\scrU)$, definable in $\mathscr A_\delta$, for player I. Before defining it we need the following.

\begin{claim}
At least one of the following holds. 
\begin{enumerate}
\item
There is a $t \in q$ so that for all $s \in q_t$ we have $L_{q_s} (\dot{X}) \in \scrU$. 
\item
There is a function $I:q \to [\omega]^{\omega}$ definable in $\mathscr A_\delta$ so that for all $t \in q$, $I(t)$ is a $q_t$-interpretation of $\dot{X}$ and $I(t) \notin \scrU$.
\end{enumerate}

\end{claim}

\begin{proof}
Assume 1 fails. Then for each $t \in q$ there is an $s \in q_t$ so that $L_{q_s}(\dot{X}) \notin \scrU$. This fact is moreover definable in $\mathscr A_\delta$ by elementarity plus the fact that the leftmost interpretation is definable. We now define (in $\mathscr A_\delta$) $I(t)$ to be $L_{q_s}(\dot{X})$ where $s \in q_t$ is of least in the lexicographic ordering so that $L_{q_s}(\dot{X}) \notin \scrU$. Clearly this satisfies the requirements.
\end{proof}

There are now two cases corresponding to whether 1 or 2 holds above. Let $q' = q_t$ if case 1 holds and $q' = q$ if case 2 holds. Note that either way $q' \in \mathscr A_\delta$, is preprocessed and has the properties we listed of $q$. Assume first that case 1 holds. The strategy is defined as follows. Player I builds the sequence $\{A_n\}_{n < \omega} \subseteq \scrU$ as needed for the game, and on the side they also build a sequence of conditions $\{q_n \; | \; n < \omega \}$ and a sequence of strictly increasing natural numbers $\{m_n\}_{n < \omega}$ so that $q_{n+1} \leq_n q_n$ for all $n < \omega$. This is done recursively as follows: at stage $0$ we let $q_0 = q'$ and, for both $0$-splitting nodes $t_0$ and $t_1$ we have that both the $q'_{t_0}$-leftmost interpretation, call it $I_0$ and the $q'_{t_1}$ leftmost interpretation call it $I_1$ are in $\scrU$. Let $A_0 = I_0 \cap I_1$. Now, player II plays a finite set $a_0 \subseteq A_0$. Let $m_0 = {\rm max}(a_0) + 1$. Let $l_0$ be the $m_0^{\rm th}$-splitting node in $l_{q'_{t_0}}$ and $l_1$ be the same for $l_{q'_{t_1}}$. By the way the left most interpretation was defined, we have that for both $i = 0$ and $i = 1$ $q'_{l_i} \forces \check{a}_0 \subseteq \dot{X}$. Let $q_1 = q'_{l_0} \cup q'_{l_1}$. It follows that $q_1 \forces \check{a}_0 \subseteq \dot{X}$. Also, note that by the assumption that we're in case 1 of the claim, $q_1$ has the property that for every $t \in q_1$, the $(q_1)_t$ leftmost interpretation is in $\scrU$. 

Now suppose we have defined $m_n$, $q_n$ and $A_n$. Assume moreover that $\{m_j\}_{j < n+1}$ is a strictly increasing sequence so that $m_j > j$ for all $j < n+1$ and $q_{j+1} \leq_j q_j$ for all $j < n$. Let $A_{n+1} = \bigcup_{t \in {\rm Split}_{n+1}(q_n)} L_{(q_n)_t}(\dot{X})$. Note that since this is the intersection of finitely many sets from $\mathscr{U}$, it too is in $\scrU$. Now player II plays some $a_{n+1} \subseteq A_n$. Let $m_{n+1} = {\rm max}\{{\rm max}(a_{n+1}), m_n\} + 1$. Note that $m_{n+1} > m_n \geq n+1$. Now for each $n+1$ splitting node $t$ of $q_n$ let $s_t$ be the $m_{n+1}^{\rm th}$ splitting node of $(q_n)_t$ in its leftmost branch. Finally let $q_{n+1} = \bigcup_{t \in {\rm Split}_{n+1}(q_n)} (q_n)_{s_t}$. Observe that again $q_{n+1} \forces \check{a}_{n+1} \subseteq \dot{X} \cap \check{m}_{n+1}$.

This completes the description of the strategy. This strategy is definable in $\mathscr A_\delta$ since all we needed to know was about the trees we define, their leftmost branches and the function $i$ described in Lemma \ref{preprocessed}. Since $\hat{\scrU}$ generates a P-point, there is a play where player I follows this strategy and still loses. Moreover, by elementarity such a play is in $\mathscr A_\delta$. Let $B = \bigcup_{n<\omega} a_n$ be the union of the elements player II plays in this play. Then $B \in \mathscr A_\delta \cap \scrU$. Let $q_\omega = \bigcap_{n < \omega} q_n$ where the $q_n$'s are the fusion sequence player I built during this play. It follows that $q_\omega \in \mathscr A_\delta$ and hence $q_\omega \in C(Y)$ and $q_\omega \forces \check{B} \subseteq \dot{X}$ so we're done.

We now turn to case 2. Fix a function $I$ as described in case 2. The proof is now the same as in case 1, except we use $I(t)$ in lieu of $L_{q_t}(\dot{X})$ and we take the intersection of the complements of each $I(t)$. The result will be that $q_\omega \forces \check{B} \cap \dot{X} = \emptyset$. 
\end{proof}

It was observed in \cite{Schilhanultrafilter} that there is a $\Delta^1_2$ base for a Ramsey ultrafilter $\mathscr{R}$ which provably consists of only constructible reals and moreover, but that this is of optimal complexity for a Ramsey ultrafilter, i.e. there is no coanalytic base for a Ramsey ultrafilter \cite[Theorem 1.4]{Schilhanultrafilter}. Since $\P_{\omega_2}$ is $\baire$-bounding, it follows from \cite[Lemma 21.12]{Hal17}, combined with the arguments presented in this section that (assuming enough elementarity of the $\mu_i$ sequence) this Ramsey ultrafilter base is preserved. It follows that we have shown the following.
\begin{corollary}
It is consistent with the existence of a $\Delta^1_3$ well-order of the reals that $2^{\aleph_0} = \aleph_2$ and there are a $\Pi^1_1$ ultrafilter base for a P-point and a (properly) $\Delta^1_2$ ultrafilter base for a Ramsey ultrafilter both of which are of size $\aleph_1$ (and of optimal complexity).
\end{corollary} 

\section{Preservation of Definable Tight MAD Families}

We now turn our attention to MAD families and the cardinal characteristic $\mfa$. Recall that an almost disjoint family $\mathcal A \subseteq [\omega]^\omega$ is {\em tight} if for each $\{X_n \; | \; n < \omega\} \subseteq \mathcal I(\mathcal A)^+$ there is a $B \in \mathcal I(\mathcal A)$ so that for all $n$ $|B \cap X_n| = \aleph_0$. Note that tightness implies maximality. We will show that there is a $\Pi^1_1$ tight MAD family in the model with the $\Delta^1_3$ well-order of \cite{FF10}. To begin we need a tight MAD family in $L$ similar to the ultrafilter $\scrU$ we preserved in the previous section. Towards this we need a coding lemma akin to Miller's coding lemma for the existence of a $\Pi^1_1$ MAD family, \cite[Lemma 8.24]{Millerpi11}. A version of this lemma was first discussed in \cite{FS21} where it was proved for eventually different sets of functions.

\begin{lemma}[Coding Lemma]
Suppose that $\mathcal X$ is a countable almost disjoint family containing a computable, infinite partition of $\omega$ into infinite sets, say $\{A_n \; | \; n < \omega\}$ and $\mathcal B$ is a countable family of infinite sets in $\mathcal I(\mathcal A)^+$. Let $Z \in [\omega]^\omega$ be arbitrary. Then there is an $X$ which is almost disjoint from every element of $\mathcal X$, has infinite intersection with every element of $\mathcal B$ and computes $Z$. Moreover $X$ can be found computably in the data.
\end{lemma}

\begin{proof}
Fix $\mathcal X$, $\{A_n\; |\; n<\omega\}$, $\mathcal B$ and $Z$ as in the statement of the lemma. We define $X$ recursively so that for each $n$, $X \cap A_n$ is even if and only if $n \in Z$. Enumerate $\mathcal B$ as $\{B_n \; | \; n< \omega\}$ so that each element appears infinitely often and let $\mathcal X \setminus \{A_n \; | \; n < \omega\} = \{X_n\; | \; n < \omega\}$

\noindent \underline{Step 0}: Let $X_0$ be the minimal element of $A_0$ if $0 \notin Z$ and the minimal two elements of $A_0$ if $0 \in Z$. 

\noindent \underline{Step n+1}: Suppose $X_n$ has been constructed, is finite, intersects each $B_0, ..., B_{n-1}$ and for all $i < n+1$, $X_n \cap A_n$ is even if and only if $n \in Z$. Now, there are two further cases.

\noindent \underline{Case 1}: $n + 1 \in Z$ and $X_n \cap A_{n + 1}$ is even or $n+1 \notin Z$ and $X_n \cap A_{n+1}$ is odd. In this case let $l$ be the minimum element of $B_n \setminus (\bigcup_{l < n+1} X_l \cup \bigcup_{l < n + 2} A_l)$.

\noindent \underline{Case 2}: $n + 1 \in Z$ and $X_n \cap A_{n + 1}$ is odd or $n+1 \notin Z$ and $X_n \cap A_{n+1}$ is even. Let $k$ be the minimum of $A_{n+1} \setminus X_n$. Then, let $l$ be as before, i.e., $l$ is the minimum element of $B_n \setminus (\bigcup_{l < n+1} X_l \cup \bigcup_{l < n + 2} A_l)$.

In both cases let $X_{n+1}$ be $X_n \cup \{l\}$ or $X_{n+1}$ be $X_n \cup\{k, l\}$ depending on the case. Finally letting $X = \bigcup_{n < \omega} X_n$ we finish the lemma. The verification that $X$ is as needed is immediate.
\end{proof}

From this we get the following.

\begin{lemma}
If $V= L$ then there is a $\Pi^1_1$ tight MAD family $\calA$ with the property that $\ZFC \proves$ ``$\calA \subseteq L$".
\end{lemma}

\begin{proof}
Assume $V=L$ and let $\leq_L$ be the definable global well-order of $L$ restricted to the reals. Note that this relation is $\Delta^1_2$. We will define a tight MAD family $\mathcal A = \{A_\alpha \; | \; \alpha < \omega_1\}$ by induction. First fix a countable, computable partition of $\omega$ into infinite sets $\{A_n \; | \; n < \omega\}$. Now suppose that $\{A_\xi \; | \; \xi < \alpha\}$ has been defined for an infinite $\alpha$. Let $\mathcal B$ be the $<_L$-least countable set of elements of $\mathcal I(\{A_\xi\; |\; \xi < \alpha \})^+$. Let $A_\alpha$ be the $X$ computed in the coding lemma for the $A_\xi$'s and $\mathcal B$ and let $Z$ a real coding $L_\alpha$. This completes the recursive definition. Clearly this can be defined by a formula in the language of set theory. Let $\varphi(x)$ be this formula. Let $\mathcal A = \{A_\alpha\; | \; \alpha < \omega_1\}$. It's clear that $\mathcal A$ is tight. It remains to see that it is $\Pi^1_1$.

To see this, let us say that a set $X \in [\omega]^\omega$ {\em codes} a structure $L_\alpha$ if the set $Z$, where $n \in Z$ if and only if $X\cap A_n$, codes a countable structure $(\omega, E)$ isomorphic to $(L_\alpha, \in)$. Observe that to say $X$ codes a structure $L_\alpha$ is $\Pi^1_1$ in the codes. Now $X \in \mathcal A$ if and only if $X$ codes a structure $L_\alpha$ and $L_{\alpha+\omega} \models \varphi(X)$. This is because, from $L_\alpha$ and $\varphi(x)$ we can retrieve $X$ (in this case) so we must have that $X \in L_{\alpha+\omega}$ and everything else is sufficiently absolute. The point is this is $\Pi^1_1$ since satisfaction is arithmetic for a fixed formula.
\end{proof}

From now on fix such a tight MAD family $\calA$. We will show that if $G \subseteq \P_{\omega_2}$ is generic over $L$ then $\calA$ is still tight (and hence MAD) in $V[G]$. 
\begin{lemma}
If $G \subseteq \P_{\omega_2}$ is generic over $L$ then in $L[G]$ we have that $\calA$ is still tight.
\label{mainlemmaMAD}
\end{lemma}

The main thrust of this proof involves the notion of a $(M, \P, \mathcal A, B)$-generic condition. Given \begin{itemize}
\item a forcing notion $\P$,
\item a tight MAD family $\calA$,
\item a condition $p \in \P$,
\item a countable model $M \prec H_{\theta}$ so that $\mathbb P, p, \mathcal A \in M$, and
\item a $B \in \mathcal I(\mathcal A)$ for which $|B\cap X| = \aleph_0$ for all $X \in \mathcal I(\mathcal A)^+ \cap M$,
\end{itemize} a condition $q\leq p$ is said to be an \emph{$(M, \mathbb P, \mathcal A, B)$-generic condition} if $q$ is $(M, \mathbb P)$-generic and $q \forces \forall \dot{Z} \in (\mathcal I(\mathcal A)^+ \cap M[\dot{G}])\, (|\dot{Z} \cap B| = \aleph_0)$. The proof of Lemma \ref{mainlemmaMAD} is by induction on $\gamma$ applied to the forcing notions $\P_\gamma$. The limit case actually follows from the preservation theorem for strongly preserving tightness of a tight MAD family in \cite[Proposition 31]{restrictedMADfamilies}. Indeed what is shown there is more local than what is stated. The proof actually shows the following.

\begin{lemma}[Essentially Proposition 31 of \cite{restrictedMADfamilies}]
Let $\langle \Q_\gamma, \dot{\mathbb R}_\gamma \; | \; \gamma \leq \alpha\rangle$ be a countable support iteration of $(S)$-proper forcing notions for some limit ordinal $\gamma$ and $p \in \Q_\gamma$. For every $\theta$ sufficiently large, $M \prec H_\theta$ countable with $p, \Q_\gamma, \gamma \in M$ and $B \in \mathcal I(\calA)$ so that $B$ has infinite intersection with every $X \in M \cap \mathcal I(\calA)^+ \cap M$ if for all $\alpha <\gamma$ with $\alpha \in M$ there is a $(M, \Q_\alpha, \calA, B)$-master condition extending $p\hook \alpha$ then there is a $(M, \Q_\gamma, \calA, B)$-master condition extending $p$.
\label{limitlemmaMAD}
\end{lemma}

We also note that the successor case preserves this property as well.

\begin{lemma}[Lemma 30 of \cite{restrictedMADfamilies}]
Assume $p_0, \Q_0, \calA \in M \prec H_\theta$, $B \in \mathcal I(\calA)$ has infinite intersection with every $X \in \mathcal I(\calA)^+ \cap M$ and $q_0 \leq p_0$ is an $(M, \Q_0, \calA, B)$ generic condition. Assume moreover that $\dot{\Q_1} \in M$ is a $\P$ name for a forcing notion, $\dot{p}_1 \in M$ is a $\Q_0$ name for a condition in $\dot{\Q}_1$ and $p_0$ forces that $\dot{q}_1$ is a $(M[\dot{G}], \dot{\Q}_1, \calA, B)$-generic condition. Then $(q_0, \dot{q}_1)$ is a $(M, \Q_1 * \dot{\Q}_1, \calA, B)$-generic condition.
\label{succlemmaMAD}
\end{lemma}

Given these two results we can now prove Lemma \ref{mainlemmaMAD}.

\begin{proof}
Assume first that $V=L$ and fix a tight MAD family $\calA$ which is $\Pi^1_1$ and provably (in $\ZFC$) consists only of constructible reals. Given a countable transitive model of a sufficient fragment of set theory, $\bar{M}$, let $A(\bar{M})$ be the $\leq_L$-least $B \in \mathcal I(\calA)$ so that for all $C \in \bar{M} \cap \mathcal I(\calA)^+$ we have that $C \cap B$ is infinite. Note that this function is $\Delta^1_3$ definable since $(\bar{M}, B) \in A$ if and only if $\bar{M}$ models a sufficient fragment of $\ZFC$ (arithmetic), $B \in \mathcal I(\calA)$ ($\Sigma^1_2$) has infinite intersection with every $C \in \calA \cap \bar{M}$ ($\Pi^1_1$) and for all $B'$ if $B' \in \mathcal I(\calA)$ has infinite intersection with every $C \in \calA \cap \bar{M}$ then $B \leq_L B'$ ($\Pi^1_2$). Since the image of any real under the Mostowski collapse is itself we will frequently refer to the image of an arbitrary countable model $M$ under $A$ by which we mean the image of $M$'s transitive collapse. We prove by induction on $\gamma < \omega_2$ that if $M\prec L_{\omega_2}$ is countable, $M \cap \omega_1 \in S$, $p \in \mathbb P_\gamma$, $\gamma, p, P_\gamma \in M$ then there is a $(M, \P_\gamma, \calA, A(M))$-generic condition $q \leq p$. Observe that this implies that for each $\gamma$ we have $\forces _\gamma$ `` $\check{\calA}$ is tight" and hence in $V^{\P_{\omega_2}}$ $\calA$ is tight. The result follows.

In light of Lemmas \ref{limitlemmaMAD} and \ref{succlemmaMAD} it suffices to restrict our attention to iterands of the form $C(Y)$. For the rest of the proof assume $V = L[Y]$ for $Y \subseteq \omega_1$ added by forcing over $L$. Fix $M$, $(p, \dot{q}_0) \in P_\gamma * \dot{C}(Y)$, etc.\ as in the inductive statement with $M \cap \omega_1 = \delta \in S$. Let $\bar{M}$ be the transitive collapse of $M$. Note that $\bar{M} = L_\alpha$ for some $\alpha$. We use the convention that if $x \in M$ then we let $\bar{x}$ be its image under the Mostowski collapse in $\bar{M}$. Let $B = A(\bar{M})$. By assumption, there is an $r \leq p$ which is a $(M, \P_\gamma, \calA, B)$-generic condition. Let $r \in G$ be $\P_\gamma$ generic over $L$. Because of how the forcing is defined, alongside the inductive assumption and the fact that $\P_\gamma$ is $S$-proper, we have that $L[G] = L[Y]$ for some $Y \subseteq \omega_1$ and in $L[Y]$ $\bar{M}[Y \cap \delta] \prec L_{\omega_2}[Y]$, $B$ has infinite intersection with each $X \in \mathcal I(\mathcal A)^+ \cap \bar{M}[Y\cap \delta]$ and $\omega_1^{\bar{M}[Y \cap \delta]} = \delta \in S$ (which of course is still stationary). This set up is the application of the inductive assumption. From now on we work in $L[Y]$.

Let $M \cap \omega_1 = \delta \in S$. It follows that $\bar{M}, \bar{M}[Y \cap \delta] \in \mathscr{A}_\delta$. Moreover, by $\Delta^1_3$-correctness of $\mathscr{A}_\delta$, we get that $B \in\mathscr{A}_\delta$. Working in $\mathscr A_\delta$ fix a countable, cofinal sequence $\{\delta_n\; | \; n < \omega\} \in \mathscr A_{\delta}$ so that for all $n< \omega$ we have $\delta_n \in \bar{M}[Y \cap \delta]$. The fact that this sequence is an {\em element} (and not just a subset) of $\mathscr A_\delta$ (though not $\bar{M}[Y \cap \delta]$) is crucial. Let $q \in M[Y]$ be the evaluation of $\dot{q}$ in $M[Y]$ and let $\bar{q}$ be its image in $\bar{M}[Y \cap \delta]$. The goal is to find a $q_\infty \leq q$ which is an $(M[Y], C(Y), \mathcal A, B)$-generic condition {\em in} $\mathscr A_\delta$(!!). 

Now, let $\{\bar{D}_k \; | \; k < \omega\}$ be an enumeration in $\mathscr{A}_\delta$ of the dense subsets of $\bar{C(Y)}$ in $\bar{M}[Y \cap \delta]$ and let $\{\dot{Z}_n \; | \; n < \omega\}$ be an enumeration in $\mathscr{A}_\delta$ of all $\bar{C(Y)}$-names for subsets of $\omega$ in $\bar{M}$ which are forced to be in $\mathcal I(\mathcal A)^+$ so that each name appears infinitely often. We will inductively define in $\mathscr{A}_\delta$ a sequence $\{\bar{q}_n \; | \; n < \omega \}$ so that the following conditions hold:
\begin{enumerate}
\item
$\bar{q} = \bar{q}_0$
\item
For all $n < \omega$ we have $\bar{q}_{n+1} \leq_n \bar{q}_n$
\item
For all $n < \omega$ $\bar{q}_{n+1} \forces \bar{D}_n \cap \bar{M} \cap \dot{G} \neq \emptyset$
\item
For all $n < \omega$ we have $|\bar{q}_{n+1}| \geq \delta_n$.
\item
For all $n < \omega$ $\bar{q}_{n+1} \forces \exists m > \check{n} \, m \in \dot{Z}_n \cap \check{B}$
\end{enumerate}

Assuming we can do this, let $q_\infty$ be the fusion of the sequence. Note that since the sequence is in $\mathscr A_\delta$, so is $q$ itself. In this case we have that $|q| \leq \delta$ by condition 4 and every branch of $q$ codes $Y$ up to $\delta$ and is therefore a condition. Given this, the fact that it witnesses the lemma is obvious.

Thus it remains to show that the sequence can be constructed in $\mathscr A_\delta$. This is done by induction. The case $n = 0$ is clear. Assume we have constructed $\{\bar{q}_l \; | \; l \leq n\}$ for some $n < \omega$ and that the sequence is in $\mathscr{A}_\delta$ and $\mathscr{A}_\delta$ inductively thinks that all of the items 1 - 5 hold. We will show there is a $\bar{q}_{n+1}$ as needed in $\mathscr A_\delta$. This will suffice since then we can choose the least such one relative to the global well ordering and this will give us a definition in $\mathscr{A}_\delta$ so that the whole sequence will be defined. 

First, let $\bar{q}_n ' \leq _n \bar{q}_n$ be a condition satisfying 2 - 4. That this is possible follows from the proof that $C(Y)$ is proper; see \cite[Lemma 7]{FF10}. Now let $\{t_i \; | \; i < 2^n\}$ enumerate the splitting nodes of $\bar{q}_n '$. Second, for each $i < 2^n$ let $Z_{i, n}$ be $\{m < \omega \; | \; (\bar{p}'_n)_{t_i} \nVdash \check{m} \notin \dot{Z}_n\}$. It's known that each $Z_{i, n} \in \mathcal I(\mathcal A)^+$. As a result $|B \cap Z_{i, n}| = \aleph_0$ and therefore we can find for each $i$ an extension $\bar{q}'_{n, i}\leq (\bar{q} '_n)_{t_i}$ and an $m > n$ so that $\bar{q}'_{n, i} \forces \check{m} \in \check{B} \cap \dot{Z}_n$. Finally let $\bar{q}_{n+1} = \bigcup_{i < 2^n} \bar{q}'_{n, i}$. This clearly works and so completes the construction and hence the lemma.

\end{proof}

As a result of these lemmas we get the following.
\begin{theorem}
If $G\subseteq\P_{\omega_2}$ is generic over $L$ then in $V[G]$ there is a $\Pi^1_1$ tight MAD family of size $\aleph_1$. In particular, it's consistent that there is a $\Pi^1_1$-tight MAD family, a $\Delta^1_3$ well order of the reals and $\mathfrak{a} = \aleph_1 < \mfc = \aleph_2$.
\label{MADthm}
\end{theorem}

\section{Preservation of Definable Selective Independent Families}
In this section we study the preservation of selective independent families. Recall that a family $\mathcal I \subseteq [\omega]^\omega$ is {\em independent} if for all finite, disjoint $\calA, \calB \in [\mathcal I]^{<\omega}$ the set $\bigcap \calA \setminus \bigcup \calB$ is infinite. Such a family is a {\em maximal independent family} if it is maximal with this property with respect to inclusion. The cardinal characteristic $\mfi$ is the least size of a maximal independent family. The following notation will facilitate our work below.
\begin{notation} For $\mathcal{I}\subseteq [\omega]^{\omega}$,
\begin{enumerate}
\item let $\mathrm{FF}(\mathcal{I})$ denote the set of finite partial functions $h$ from $\mathcal{I}$ to  $\{0,1\}$, and
\item for $h\in\mathrm{FF}(\mathcal{I})$  write $\mathcal{I}^h$ for $$\bigcap\{A\mid A\in\mathrm{dom}(h)\textnormal{ and }h(A)=1\}\cap\bigcap\{\omega\backslash A\mid A\in\mathrm{dom}(h)\textnormal{ and }h(A)=0\}.$$
\end{enumerate}
\end{notation}

In this terminology, a family $\mathcal{I}\subseteq [\omega]^\omega$ is independent if $\mathcal{I}^h$ is infinite for all $h\in\mathrm{FF}(\mathcal{I})$ and an independent family $\mathcal{I}$ is maximal if $\forall X\in [\omega]^\omega\;\exists h\in\mathrm{FF}(\mathcal{I})\text{ such that }\mathcal{I}^h\cap X\text{ or }\mathcal{I}^h\backslash X\text{ is finite}$. Such an $\mathcal{I}$ is \emph{densely maximal} if $\forall X\in [\omega]^\omega\text{ and }h'\in\mathrm{FF}(\mathcal{I})\;\exists h\supseteq h'\text{ in }\mathrm{FF}(\mathcal{I})\text{ such that }\mathcal{I}^h\cap X\text{ or }\mathcal{I}^h\backslash X\text{ is finite.}$

The \emph{density ideal of $\mathcal{I}$}, denoted $\mathrm{id}(\mathcal{I})$, is $$\{X\subseteq\omega\mid \forall h'\in\mathrm{FF}(\mathcal{I})\;\exists h\supseteq h'\text{ in }\mathrm{FF}(\mathcal{I})\text{ such that }\mathcal{I}^h\cap X\text{ is finite}\}.$$ Dual to the density ideal of $\mathcal{I}$ is the \emph{density filter of $\mathcal{I}$}, denoted $\mathrm{fil}(\mathcal{I})$ and defined as $$\{X\subseteq\omega\mid \forall h'\in\mathrm{FF}(\mathcal{I})\;\exists h\supseteq h'\text{ in }\mathrm{FF}(\mathcal{I})\text{ such that }\mathcal{I}^h\backslash X\text{ is finite}\}.$$ Observe that for an infinite independent family $\mathcal{I}$, none of the above definitions' meanings change if we replace the word ``finite'' with ``empty''. The key definition for this section is the following.

\begin{definition}
An independent family $\mathcal I$ is called {\em selective} if it is densely maximal and ${\rm fil}(\mathcal I)$ is Ramsey.
\end{definition}

It was shown in \cite{DefMIF} that if $V= L$, then there is a $\Sigma^1_2$ selective independent family completely contained in $L$ and, moreover, the existence of a $\Sigma^1_2$ maximal independent family implies the existence of a $\Pi^1_1$-maximal independent family. It follows that preserving a fixed $\Sigma^1_2$ selective independent family contained in $L$ will allow us to obtain the existence of a $\Pi^1_1$ maximal independent family. Let us fix a $\Sigma^1_2$ selective independent family contained in $L$, say $\mathcal I$, for the rest of this section. We will show the following.

\begin{theorem}
If $G\subseteq \P_{\omega_2}$ is generic over $L$ then in $L[G]$ we have that $\mathcal I$ is selective (and hence maximal). Consequently the existence of a $\Delta^1_3$ well-order is consistent with a coanalytic maximal independent family of size $\aleph_1 < 2^{\aleph_0} = \aleph_2$.
\label{mainthmIND}
\end{theorem}

The centerpiece of the proof of this theorem is the application of a preservation result of Shelah from \cite{Sh92} to a proof of the fact that $C(Y)$ preserves the selectivity of sufficiently definable selective independent families. Towards this we need some more on independent families.

In our discussion of dense independence, we will make use of the following Lemma. The equivalence of $(1)$ and $(2)$ can be found in~\cite[Lemma 31]{FM20}, the of $(2)$ and $(3)$ implicitly appears in~\cite[Theorem 29]{FM19}, as well as~\cite{Sh92}. See also~\cite[Lemma 2]{VFJS21}. For completeness, we give a detailed proof.

\begin{lemma} The following are equivalent:
	\begin{enumerate}
		\item $\mathcal{I}$ is a densely maximal independent family.
		\item For all  $h\in\mathrm{FF}(\mathcal{I})$ and all $X\subseteq\mathcal{I}^h$ either $\mathcal{I}^h\backslash X\in\mathrm{id}(\mathcal{I})$ or there is $h'\in\mathrm{FF}(\mathcal{I})$ such that $h'\supseteq h$ and $\mathcal{I}^{h'}\subseteq\mathcal{I}^h\backslash X$.
		\item For each $X\in\mathcal{P}(\omega)$ there is $h\in\mathrm{FF}(\mathcal{I})$ such that $X\subseteq\omega\backslash\mathcal{I}^h$.
	\end{enumerate}
\end{lemma}
\begin{proof}
To show that $(1)$ implies $(2)$, consider $h\in\mathrm{FF}(\mathcal{I})$ and $X\subseteq\mathcal{I}^h$. Suppose $\mathcal{I}^h\backslash X\notin\mathrm{id}(\mathcal{I})$. Thus, there is $h'\in\mathrm{FF}(\mathcal{I})$ such that for all $h''\supseteq h'$ the set $\mathcal{I}^{h''}\cap(\mathcal{I}^h\backslash X)$ is non-empty. Note that if $h\perp h'$, then $\mathcal{I}^{h'}\cap (\mathcal{I}^h\backslash X)=\emptyset$, which is a contradiction. Therefore $h\not\perp h'$. Without loss of generality $h'\supseteq h$ and so for all $h''\supseteq h'$, $\mathcal{I}^{h''}\backslash X\neq\emptyset$. Since the family $\mathcal{I}$ is densely maximal, there is $h''\supseteq h'$ such that $\mathcal{I}^{h''}\cap X=\emptyset$. Thus, $\mathcal{I}^{h''}\subseteq \mathcal{I}^h\backslash X$. 

To see that $(2)$ implies $(3)$ consider any $Z\notin\mathrm{fil}(\mathcal{I})$. Then $\omega\backslash Z\notin\mathrm{id}(\mathcal{I})$ and so there is $h\in\mathrm{FF}(\mathcal{I})$ such that $h'\supseteq h$, and $|\mathcal{I}^{h'}\cap (\omega\backslash Z)|=|\mathcal{I}^{h'}\backslash Z|=\omega$. Let $Y=\mathcal{I} ^h\backslash Z$. Thus, $Y\subseteq\mathcal{I}^h$. By part $(2)$ either $\mathcal{I}^h\backslash Y\in\mathrm{id}(\mathcal{I})$ or there is $h'\supseteq h$ such that $\mathcal{I}^{h'}\subseteq\mathcal{I}^h\backslash Y$. Suppose $\mathcal{I}^h\backslash Y\in\mathrm{id}(\mathcal{I})$. Then there is $h'\supseteq h$ such that $\mathcal{I}^{h'}\cap (\mathcal{I}^h\backslash Y)=\mathcal{I}^{h'}\backslash Y=\emptyset$. However, $\mathcal{I}^{h'}\backslash Y=\mathcal{I}^{h'}\cap Z=\emptyset$ and so $Z\subseteq\omega\backslash\mathcal{I}^{h'}$ and we are done. If $\exists h'\supseteq h$ such that $\mathcal{I}^{h'}\subseteq\mathcal{I}^h\backslash Y=\mathcal{I}^h\cap Z$, then $\mathcal{I}^{h'}\cap (\omega\backslash Z)=\mathcal{I}^{h'}\backslash Z=\emptyset$, which is a contradiction to the choice of $h$.

Next, we show that $(3)$ implies $(2)$. Let $X\subseteq\mathcal{I}^h$ for some $h\in\mathrm{FF}(\mathcal{I})$. Consider $Y=\mathcal{I}^h\backslash X$. If $\omega\backslash Y\in\mathrm{fil}(\mathcal{I})$, then $Y=\mathcal{I}^h\backslash X\in\mathrm{id}(\mathcal{I})$. Otherwise there is $h^*$ such that $\omega\backslash Y\subseteq\omega\backslash \mathcal{I}^{h^*}$, which implies that $\mathcal{I}^{h^*}\subseteq Y\subseteq \mathcal{I}^h\backslash X$. Note that if $h\perp h^*$, then for some $C\in\mathcal{I}$ we have (without loss of generality) that $\mathcal{I}^{h^*}\subseteq C$ and $\mathcal{I}^{h}\subseteq\omega\backslash C$, which contradicts $\mathcal{I}^{h^*}\subseteq \mathcal{I}^h$. Thus $h^*\not\perp h$ and so $\mathcal{I}^{h*\cup h}\subseteq \mathcal{I}^{h^*}\subseteq\mathcal{I}^h\backslash X$. 

Finally, we show that $(2)$ implies $(1)$. Let $X\in [\omega]^\omega\backslash \mathcal{I}$ and let $h\in\mathrm{FF}(\mathcal{I})$. We want to show that there is $h'\supseteq h$ such that either $\mathcal{I}^{h'}\cap X=\emptyset$ or $\mathcal{I}^{h'}\backslash X=\emptyset$. Consider the set $Y=X\cap\mathcal{I}^h$. Thus, $Y\subseteq\mathcal{I}^h$. If $\mathcal{I}^h\backslash Y\in\mathrm{id}(\mathcal{I})$, then $\mathcal{I}^h\backslash X\in\mathrm{id}(\mathcal{I})$ and so there is $h'\supseteq h$ such that $\mathcal{I}^{h'}\cap (\mathcal{I}^h\backslash X)=\mathcal{I}^{h'}\backslash X=\emptyset$. Otherwise, there is $h'\supseteq h$ such that $\mathcal{I}^{h'}\subseteq\mathcal{I}^h\backslash Y$ and so $\mathcal{I}^{h'}\cap Y=\emptyset$. However, $\mathcal{I}^{h'}\cap Y=\mathcal{I}^{h'}\cap (X\cap \mathcal{I}^h)=\mathcal{I}^{h'}\cap X=\emptyset$.
\end{proof}

In particular, we obtain:

\begin{lemma} 
	A family $\mathcal{I}\subseteq [\omega]^\omega$ is densely maximal if and only if $$P(\omega)=\mathrm{fil}(\mathcal{I})\cup\langle\omega\backslash\mathcal{I}^h\mid h\in\mathrm{FF}(\mathcal{I})\rangle_{\mathrm{dn}}.$$
	\label{lemma0i}
\end{lemma}

The following are easily verified.

\begin{lemma}\hfill
\begin{enumerate}
\item
$\mathcal{I}\subseteq\mathcal{I}'$ implies that $\mathrm{fil}(\mathcal{I})\subseteq\mathrm{fil}(\mathcal{I}')$;
\item 
if $\kappa$ is a regular uncountable cardinal and $\langle\mathcal{I}_\alpha\mid\alpha<\kappa\rangle$ is a continuous increasing chain then $\mathrm{fil}(\bigcup_{\alpha<\kappa}\mathcal{I}_\alpha)=\bigcup_{\alpha<\kappa}\mathrm{fil}(\mathcal{I}_\alpha)$;
\item 
$\mathrm{fil}(\mathcal{I})=\bigcup\{\,\mathrm{fil}(\mathcal{J})\mid \mathcal{J}\in [\mathcal{I}]^{\leq\omega}\}$.
\end{enumerate}
\label{lemma1i}
\end{lemma}

Our goal is to understand how densely maximal families behave with respect to forcing notions of the form $C(Y)$. As in the previous two sections, we assume from now on, unless explicitly stated otherwise, that $V = L[Y]$ for some $Y \subseteq \omega_1$ added by forcing over $L$. Our first more substantial lemma is the following. 
\begin{lemma}
Let $\mathcal{I}$ be an independent family. Let $H$ be $C(Y)$-generic. Then $\mathrm{fil}(\mathcal{I})^{V[H]}$ is generated by $\mathrm{fil}(\mathcal{I})^V$. In other words, for each $X\in\mathrm{fil}(\mathcal{I})^{V[H]}$ there is a $Z\in\mathrm{fil}(\mathcal{I})^V$ with $Z\subseteq X$.
\label{lemma2i} 
\end{lemma}

\begin{proof}[Proof of Lemma \ref{lemma2i}] 
Fix $p\in C(Y)$ such that $p\Vdash\dot{X}\in\mathrm{fil}(\mathcal{I})$.
\begin{claim} There exists a $q\leq p$ and a countable $\mathcal{J}\subseteq\mathcal{I}$ in $V$ such that $q\Vdash\dot{X}\in\mathrm{fil}(\mathcal{J})$.
\label{claim1i}
\end{claim}

\begin{proof}[Proof of Claim \ref{claim1i}] 
Fix in $V$ an enumeration $e:\mathcal{I}\to\mathrm{Ord}$; by Lemma \ref{lemma1i}, $p\forces \dot{X}\in\mathrm{fil}(e^{-1}(\dot{E}))$ for some countable $\dot{E}\subseteq\mathrm{Ord}$. As $C(Y)$ is proper, there exists a $q\leq p$ and countable $F$ in $V$ so that $q\Vdash\dot{E}\subseteq F$. Let $\mathcal{J}=e^{-1}(F)$; it follows that $q\Vdash\dot{X}\in\mathrm{fil}(\mathcal{J})$, as desired.
\end{proof}

Next, identifying $\mathrm{FF}(\mathcal{J})$ with $2^{<\omega}$ (since $\mathcal J$ is countable and $\mathrm{FF}(\mathcal J)$ is the set of finite functions on it), let $\dot{D}$ denote a $C(Y)$-name for the dense open subset of $2^{<\omega}$ defined by $q\Vdash ``\tau\in\dot{D}$ if and only if $\mathcal{J}^\tau\backslash\dot{X}=\emptyset$''.
\begin{claim}\label{claim2}
Let $\Delta$ denote the set of $p\in C(Y)$ for which there exists in $V$ a dense $K\subseteq 2^{<\omega}$ such that $p\Vdash K\subseteq\dot{D}$. Then $\Delta$ is dense below $q$.
\end{claim}
The above two claims together imply that for any $C(Y)$-generic $H$, there exists in $V$ a dense $K\subseteq 2^{<\omega}$ such that $V[H]\vDash K\subseteq \dot{D}_H$; such a $K$ then entails the following: 
\begin{itemize}
\item $Z:=\bigcup_{\tau\in K}\,\mathcal{J}^{\tau}$ is an element of $\mathrm{fil}(\mathcal{I})^V$, and
\item $V[H]\vDash Z\subseteq \dot{X}_H$.
\end{itemize}
The proof of Claim \ref{claim2} will therefore complete the proof of the lemma.

\begin{proof}[Proof of Claim \ref{claim2}] Let $\vec{\sigma}=\langle\sigma_n\mid n\in\omega\rangle$ recursively enumerate $2^{<\omega}$. Fix an $r\leq q$ and a countable elementary submodel $M$ of $L_{\omega_2}[Y]$ containing $r$ and $C(Y)$. Let $\delta=M\cap\omega_1$. We will extend $r$ to a condition $s\in C(Y)$ with $|s|=\delta$ which decides a dense subset of $\dot{D}$; instrumental for this purpose will be the existence of a sequence $\langle\delta_n\mid n\in\omega\rangle$ in $\mathcal{A}_\delta:=L_{\mu_\delta}[Y\cap\delta]$ which is cofinal in $\delta$. More precisely, within $\mathcal{A}_\delta$ we will inductively construct a fusion sequence $\vec{r}=\langle r_n\mid n\in\omega\rangle$ of elements of $C(Y)\cap M$ below $r$ together with a family $K=\{\tau_n\mid n\in\omega\}$ of elements of $2^{<\omega}$ such that $\sigma_n\subseteq \tau_n$ for all $n\in\omega$. The branches of the conditions $r_n$ will all code $Y$ below $\delta_n$; therefore, writing $r_\infty$ for the fusion of $\vec{r}$, the construction will ensure that $r_\infty = s$ is a condition of $C(Y)$ which forces ``$K\subseteq\dot{D}$'', as desired.

Begin by extending $r$ to an $r_0\in M$ and $\sigma_0$ to a $\tau_0$ such that $|r_0|\geq\delta_0$ and $r_0$ forces ``$\tau_0\in\dot{D}$''. More generally, our inductive assumption at stage $n$ will be that we have constructed a condition $r_n\in M$ whose branches all code $Y$ below $\delta_n$ and a $\tau_n\supseteq\sigma_n$ such $r_n\Vdash \tau_n\in\dot{D}$. We describe the passage to stage $n+1$; all conditions arising in this passage should be understood to be elements of $M$. First, choose an $r_n'\leq_n r_n$ such that $|r'_n|\geq\delta_{n+1}$. Let $\langle u_j\mid j\in 2^{n+1}\rangle$ enumerate $\mathrm{Split}_{n+1}(r'_n)$. Since $(r'_n)_{u_0}\Vdash ``\dot{D}\text{ is dense''}$ there exists an $s_0\leq (r'_n)_(u_0)$ and $\tau'_0\supseteq\sigma_{n+1}$ such that $s_0\Vdash \tau'_0\in\dot{D}$. Similarly, there will exist an $s_1\leq( r'_n)_{u_1}$ and a $\tau'_1\supseteq\tau_0'$ such that $s_1\Vdash\tau_1'\in\dot{D}$. Continuing in this fashion, we construct $\tau'_0\subseteq\dots\subseteq \tau'_j\subseteq\dots\subseteq \tau'_{2^{n+1}-1}$ and conditions $s_j\leq (r'_n)_{u_j}$ such that $s_j$ forces ``$\tau'_j\in\dot{D}$'' for each $j\leq 2^{n+1}-1$. Let $\tau_{n+1}=\tau'_{2^{n+1}-1}$ and let $r_{n+1}=\bigcup_{j<2^{n+1}}s_j$. Observe that we have conserved our induction hypothesis; letting $r_\infty$ be the fusion of the conditions $r_n$ then completes the argument in the manner described.
\end{proof}
\end{proof}

Recall the following preservation result, due to Shelah.
\begin{theorem}[Shelah, See Conclusion 2.15D, pg. 305 of \cite{PIP}, see also \cite{FM19}, Theorem 27]
If $\langle \Q_\alpha, \dot{\mathbb R}_\alpha \; | \; \alpha < \gamma\rangle$ is a countable support iteration of proper forcing notions so that for all $\alpha < \gamma$ $\forces_{\alpha}$ ``$\dot{\mathbb R}_\alpha$ is proper and any new dense open subset of $2^{<\omega}$ contains an old one" then $\Q_\gamma$ has the property that any new open sense subset of $2^{<\omega}$ contains an old one.
\label{DOiterationthm}
\end{theorem}

We now have the following lemma.
\begin{lemma} 
Let $\mathcal{I}$ be an independent family. For every $\alpha\leq \omega_2$, the filter $\mathrm{fil}(\mathcal{I})^{V^{\mathbb{P}_\alpha}}$ is generated by $\mathrm{fil}(\mathcal{I})^V$. In other words,
$V^{\mathbb{P}_\alpha}\vDash \,\forall X\in\mathrm{fil}(\mathcal{I})\,\exists Z\in \mathrm{fil}(\mathcal{I})\cap V\text{ such that }Z\subseteq X\,.$
\label{preservation1}
\end{lemma}

\begin{proof} 
Immediate from Lemma \ref{lemma2i} and Theorem \ref{DOiterationthm}.
\end{proof}

Given all of these preliminaries we arrive now to our main preservation result for independent families. First, recall from \cite[Lemma 3.2]{Sh92} the following result:
\begin{theorem}\label{Shelah_preservation}
Assume \textsf{CH}. Let $\delta$ be a limit ordinal and let $\langle\mathbb{P}_\alpha,\dot{\mathbb{Q}}_\beta\mid\alpha\leq\delta,\beta<\delta\rangle$ be a countable support iteration of ${^\omega}\omega$-bounding proper posets. Let $\mathcal{F}\subseteq P(\omega)$ be a Ramsey set and let $\mathcal{H}$ be a subset of $P(\omega)\backslash\langle\mathcal{F}\rangle_{\mathrm{up}}$. If $V^{\mathbb{P}_\alpha}\vDash P(\omega)=\langle\mathcal{F}\rangle_{\mathrm{up}}\cup\langle\mathcal{H}\rangle_{\mathrm{dn}}$ for all $\alpha<\delta$ then $V^{\mathbb{P}_\delta}\vDash P(\omega)=\langle\mathcal{F}\rangle_{\mathrm{up}}\cup\langle\mathcal{H}\rangle_{\mathrm{dn}}$ as well.
\end{theorem}
\begin{remark}
This holds even with ``proper'' weakened to ``$S$-proper''; in particular, it holds for $\P_{\omega_2}$.
\end{remark}

Using this theorem we are almost done with the proof of Theorem \ref{mainthmIND}. Indeed this will cover the limit step of the preservation theorem, therefore we need to consider the successor step. 
\begin{lemma}
Assume $V=L[Y]$ with $Y \subseteq  \omega_1$ added by a forcing in $L$. If in $V$ we have that $\mathcal I$ is densely maximal then $\forces_{C(Y)}$ ``$\check{\mathcal I}$ is densely maximal".
\label{INDCYlemma}
\end{lemma}

Before beginning let us note that in $L[Y]$ the set $\mathrm{fil}(\mathcal I)$ is Ramsey. To see this note first that, by $\CH$ in the ground model ${\rm fil}(\mathcal I)$ is generated by a $\subseteq^*$-decreasing sequence of size $\omega_1$, say $\{A_\gamma \; | \; \gamma < \omega_1\}$ and hence by Lemma \ref{preservation1} given any family of countably many $\{X_n\}_{n < \omega} \subseteq {\rm fil}(\mathcal  I)$ there is an $A_\gamma$ almost contained in all of them so ${\rm fil}(\mathcal I)$ is a P-set. Moreover by the fact that the forcing is $\baire$-bounding it is a Q-set as well.

\begin{proof}
Work in $L[Y]$ and assume for contradiction that there exists a $C(Y)$-name $\dot{X}$ for such an $X$, together with a $p\in C(Y)$ such that
\begin{align}\label{pline}p\Vdash ``\dot{X}\notin\mathrm{fil}(\mathcal{I})\text{ and }|\dot{X}\cap\mathcal{I}^h|=\aleph_0\textnormal{ for all }h\in \mathrm{FF}(\mathcal{I})\text{''}.\end{align}
Applying Lemma \ref{preprocessed} we can find a $q\leq p$ in $C(Y)$ and a family $\mathbf{x}(q):=\langle x_t\mid t\in \mathrm{Split}(q)\rangle$ such that for all $t\in\mathrm{Split}_n(q)$,
$$q_t\Vdash``x_t\text{ is the characteristic function of }\dot{X}\cap n\text{''}.$$ Moreover, again by Lemma \ref{preprocessed}, we may, if desired, also assume that $\mathbf{x}(q)\in \mathcal{A}_{\delta}$ with $|q| = \delta$. 

\begin{claim}
For each $t\in\mathrm{Split}(q)$, let $Z_t=\{m\in\omega\mid q_t \not\Vdash m\not\in\dot{X}\}$. Then $$Z_t=\bigcup_{s\sqsupseteq t} x_s^{-1}(1).$$ In consequence, any model $\mathcal{A}_\delta$ containing $q$ and $\mathbf{x}(q)$ can compute $Z_t$, for all $t\in\mathrm{split}(q)$.
\label{keyclaimIND}
\end{claim}
\begin{proof}[Proof of Claim \ref{keyclaimIND}] If $m\in\bigcup_{s\sqsupseteq t} x_s^{-1}(1)$ then $x_s(m)=1$ for some $s\sqsupseteq t$; it then follows from $q_s\leq q_t$ and $q_s\Vdash m\in\dot{X}$ that $m\in Z_t$. To see the reverse inclusion, namely that $$Z_t\subseteq\bigcup_{s\sqsupseteq t} x_s^{-1}(1),$$ suppose that $r\leq q_t$ forces ``$m\in\dot{X}$''; suppose also that $t\in\mathrm{Split}_n(q)$ for some $n\leq m$ (for $m<n$ the inclusion is clear, since $m\in Z_t$ if and only if $x_t(m)=1$). Then for all $s\in r\cap\mathrm{Split}_{m+1}(q)$, we must have $x_s(m)=1$. For if this were not the case for some $s$ then $q_s$, and hence $r_s$, would force ``$m\not\in\dot{X}$'' --- but since $r_s\leq r$, this would entail contradiction.
\end{proof}

\begin{claim}\label{lemmahullsinfilter}
$Z_t\in\mathrm{fil}(\mathcal{I})$ for all $t\in\mathrm{Split}(q)$.
\end{claim}
\begin{proof}[Proof of Claim \ref{lemmahullsinfilter}]
Suppose for contradiction that the lemma fails for some $Z_t$. Then by our inductive assumption together with Lemma \ref{lemma0i}, there exists an $h\in\mathrm{FF}(\mathcal{I})$ with $\mathcal{I}^h\cap Z_t=\emptyset$. But since $q_t\leq p$ and $q_t\Vdash \dot{X}\subseteq Z_t$, this contradicts line \ref{pline} above. 
\end{proof}
As noted before the beginning of this proof  $\mathrm{fil}(\mathcal{I})$ is a P-set  (again this also follows immediately from Lemma \ref{preservation1}, together with the fact that, by construction, $\mathcal{I}$ has an $\subseteq^*$-decreasing filterbase of ordertype $\omega_1$). It then follows from Claim \ref{lemmahullsinfilter} that there exists a pseudo-intersection $B\in\mathrm{fil}(\mathcal{I})$ of $\{Z_t\mid t\in\mathrm{Split}(q)\}$. Moreover, by elementarity we can assume that such a $B$ exists in $\mathcal A_\delta$.

Any such $B$ determines a function $f:\omega\to\omega$ given by $$n\mapsto\mathrm{max}\;\big(B\backslash\cap\{Z_t\mid t\in\mathrm{Split}_j(q)\textnormal{ and }j\leq n+1\}\big).$$
Since $\mathrm{fil}(\mathcal{I})$ is a Q-set there will then exist a $C\subseteq B$ in $\mathrm{fil}(\mathcal{I})$ --- and hence, again by elementarity, a $C\subseteq B$ in $\mathrm{fil}(\mathcal{I})\cap\mathcal{A}_{\delta}$ --- such that if $\langle k(n)\mid n\in\omega\rangle$ enumerates $C\cup\{0\}$ then
$$f(k(n))<k(n+1)$$
for all $n\in\omega$; without loss of generality $f(1)<k(1)$ as well.

We describe now how to construct from such a $k$, $q$, and $\mathbf{x}(q)$ in $\mathcal{A}_{\delta}$ an $r\in\mathcal{C}(Y)$ below $q$ which forces ``$C\subseteq\dot{X}$''; observe that the existence of such an $r$ contradicts our assumptions about $\dot{X}$, thereby concluding the proof.


For each $t\in\mathrm{Split}_0(q)$ and $i\in\{0,1\}$ extend $t^\frown i$ to a $s(t,i)\in\mathrm{Split}_{k(1)}(q)$ forcing ``$k(1)\in\dot{X}$'' (such exists by construction; for uniformity, here and throughout take always the leftmost such $s(t,i)$). Let $q_1=\bigcup_{i\in 2} q_{s(t,i)}$ and note that $q_1\leq_0 q_0=q$. Continue in this fashion, at stage $n$ choosing a $s(t,i)\in q_n\cap\mathrm{Split}_{k(n+1)}(q)$ extending $t^\frown i$, for each $t\in\mathrm{Split}_n(q_n)$ and $i\in\{0,1\}$. Let $$q_{n+1}=\bigcup_{(t,i)\in\mathrm{Split}_n(q_n)\times 2}(q_n)_{s(t,i)}$$ and observe that $q_{n+1}\leq_n q_n$ and that the fusion $r$ of the $q_n$s so defined forces ``$C\in\dot X$'', as desired. Observe lastly that $r$ is constructible from $k$, $q$, and $\mathbf{x}(q)$, ensuring that $r\in\mathcal{A}_{\delta}$ and hence that $|r|\leq\delta$; this ensures that $r$ is indeed a condition of $\mathcal{C}(Y)$.
\end{proof} 

Given the proof of Lemma \ref{INDCYlemma}, we can now finish the proof of Theorem \ref{mainthmIND}.

\begin{proof}[Proof of Theorem \ref{mainthmIND}]
The proof is by induction on $\alpha < \omega_2$. There are two cases, depending on whether $\alpha$ is a limit or a successor ordinal. As noted above, $\mathrm{fil}(\mathcal I)$ is Ramsey in $V^{\P_\alpha}$ for all $\alpha \leq \omega_2$ so we need to ensure simply that $\mathcal I$ remains densely maximal.

\underline{Case 1: $\alpha = \xi + 1$ for some $\xi$}. This follows from Lemma \ref{INDCYlemma} noting that no other step adds reals.

\underline{Case 2: $\alpha$ is a limit ordinal}. We have already discussed that ${\rm fil}(\mathcal I)$ remains Ramsey throughout the iteration so it remains to see that $\mathcal I$ is densely maximal in $V^{\P_\alpha}$. Observe first that for all $\beta\leq\alpha$, $$V^{\mathbb{P}_\beta}\vDash ``\mathrm{fil}(\mathcal{I})=\langle\mathrm{fil}(\mathcal{I})\cap V\rangle_{\mathrm{up}}\text{''},$$
by Lemma \ref{preservation1} above. By Lemma \ref{lemma0i}, our inductive assumption then takes the form
$$V^{\mathbb{P}_\beta}\vDash P(\omega)=\langle\mathrm{fil}(\mathcal{I})\cap V\rangle_{\mathrm{up}}\cup\langle\omega\backslash\mathcal{I}^h\mid h\in\mathrm{FF}(\mathcal{I})\rangle_{\mathrm{dn}}$$
for all $\beta<\alpha$. By Theorem \ref{Shelah_preservation},
$$V^{\mathbb{P}_\alpha}\vDash P(\omega)=\langle\mathrm{fil}(\mathcal{I})\cap V\rangle_{\mathrm{up}}\cup\langle\omega\backslash\mathcal{I}^h\mid h\in\mathrm{FF}(\mathcal{I})\rangle_{\mathrm{dn}}$$ must hold as well. Again apply Lemma \ref{lemma0i} to conclude that $\mathcal{I}$ is densely maximal in $V^{\mathbb{P}_\alpha}$.
\end{proof}

\section{Applications and Models}
Having proved these three preservation theorems we can now turn to applications. The first is an immediate consequence of of Theorems \ref{ppointthm}, \ref{MADthm} and \ref{mainthmIND}.
\begin{theorem}
Let $G \subseteq \P_{\omega_2}$ be generic over $L$. In $L[G]$ we have $\mfa = \mfu = \mfi = \aleph_1 < 2^{\aleph_0} = \aleph_2$, with each cardinal characteristic witnessed by a $\Pi^1_1$ set and a $\Delta^1_3$ well-order.
\label{alltogetherthm}
\end{theorem}

Note that since $C(Y)$ has the Sacks property, in $L[G]$ we have that $cof(\Null)$, and hence all the cardinals in Cicho\'{n}'s diagram are $\aleph_1$ as well. 

We can also mix and match the cardinal characteristics while still preserving coanalytic witnesses.

\begin{theorem}
The following are consistent with $2^{\aleph_0} = \aleph_2$ and a $\Delta^1_3$ well-order of the reals.
\begin{enumerate}
\item
$\mfa = \mfu = \aleph_1 < \mfi = \aleph_2$ and there are a $\Pi^1_1$ tight MAD family of size $\aleph_1$ and a $\Pi^1_1$ ultrafilter base for a P-point of size $\aleph_1$.
\item
$\mfa = \mfi  = \aleph_1 < \mfu = \aleph_2$ and there are a $\Pi^1_1$ tight MAD family of size $\aleph_1$ and a $\Pi^1_1$ maximal independent family of size $\aleph_1$.
\item
$\mfa =\aleph_1 < \mfi = \mfu = \aleph_2$ and there is a $\Pi^1_1$ tight MAD family. 
\end{enumerate}
\label{cardchar}
\end{theorem}

\begin{proof}
Each of these follows from our preservation theorems alongside the right choice of ``woven in" forcing for $\dot{\Q}^0_\alpha$. For the first one we can use Miller forcing, which, since it increases $\mfd$, will force $\mfi = 2^{\aleph_0}$. However Miller forcing preserves P-points \cite[Lemma 25.5]{Hal17} and strongly preserves the tightness of any MAD family \cite{restrictedMADfamilies} so the coanalytic witnesses to $\mfa = \mfu = \aleph_1$ will be preserved. For the second model we force with Shelah's forcing $\Q_\mathcal I$ from \cite{Sh92} alongside an appropriate bookkeeping device. This forcing will preserve a selective independent family, strongly preserve the tightness of any tight MAD family (\cite{CFGS21}) but increases $\mfu$. For the final model we alternate between Miller forcing and $\Q_\mathcal I$ forcing.
\end{proof}

\section{Conclusion and Open Questions}
In light of Theorem \ref{alltogetherthm}, it is reasonable to ask about whether in $V^{\P_{\omega_2}}$ there are optimal complexity witnesses to ``maximal" sets of reals of size $\aleph_2$.

\begin{question}
What is the lowest projective complexity of a MAD family, a maximal independent family or an ultrafilter base of size $\aleph_2$ in $V^{\P_{\omega_2}}$? Could there be co-analytic witnesses for any of these? What about in the models discussed in Theorem \ref{cardchar}?
\end{question}

We can also ask about the possibility of coanalytic witnesses to $\mfa$, $\mfi$ and $\mfu$ in a model of $2^{\aleph_0} > \aleph_2$ alongside a $\Delta^1_3$ well-order. Such a model was constructed in \cite{FFZ11}.

\begin{question}
Is it consistent that $2^{\aleph_0} > \aleph_2$ alongside a $\Delta^1_3$ well-order of the reals and $\mfa = \mfu = \mfi = \aleph_1$ all with coanalytic witnesses? 
\end{question}
The model of \cite{FFZ11} was built using finite support iteration of ccc posets, entailing that $\mfi = 2^{\aleph_0}$. It seems difficult to avoid increasing $\mfi$ while forcing $2^{\aleph_0} > \aleph_2$ in this way.

\end{document}